\documentclass[11pt, reqno]{amsart}

\usepackage{bbm,bm}
\usepackage{amssymb}

\usepackage{color}
\usepackage{pstricks}
\usepackage{graphicx} 

	\definecolor{myred}{rgb}{.5,.1,.1}
	\definecolor{mygreen}{rgb}{.1,.5,.1}
	\definecolor{myblue}{rgb}{.1,.1,.5}
	\definecolor{mycyan}{cmyk}{.6,0,0,0}
	\definecolor{mymagenta}{cmyk}{0,.6,0,0}

    \definecolor{matjaz}{rgb}{.2,.6,.2}
    \definecolor{aaron}{rgb}{.8,.2,.2}

\usepackage{ytableau}

\usepackage[bookmarks=false,colorlinks,linkcolor=myblue,citecolor=mygreen]{hyperref}

\usepackage{subfigure}
\usepackage[width=.85\textwidth,font=small,labelfont=sc]{caption}

\PII{}
\copyrightinfo{}{}

\theoremstyle{plain}
\newtheorem{theorem}{Theorem}

\newtheorem{corollary}[theorem]{Corollary}
\newtheorem{lemma}[theorem]{Lemma}
\newtheorem{proposition}[theorem]{Proposition}

\theoremstyle{definition}

\theoremstyle{remark}
\newtheorem*{remark}{Remark} 

\newcommand{\demph}[1]{\textcolor{myblue}{\it #1}}

\def\qor{\quad\hbox{or}\quad}
\def\qand{\quad\hbox{and}\quad}

\def\field{\Bbbk}

\newcommand{\bQ}{\mathbb{Q}}
\newcommand{\bZ}{\mathbb{Z}}

\newcommand{\sym}[1][0]{\if0#1{\Lambda}\else{\Lambda[#1]}\fi} 

\def\harpoon{\rightharpoonup}

\def\innp(#1,#2){\left\langle#1,#2\right\rangle}
\def\innpt(#1,#2){\left\langle#1,#2\right\rangle_t}

\newcommand{\skewshape}[3][8]{%
	{\ytableausetup{aligntableaux=center,boxsize=#1pt}
	\ydiagram[*(gray)]{#3}
	*[*(white)]{#2}}
}

\newcommand{\qbinom}[3]{\genfrac{[}{]}{0pt}{}{#1}{#2}_{#3}} 

\newcommand{\m}[2]{m_{#1\!}\left(#2\right)} 

\DeclareMathOperator{\rib}{rib}
\DeclareMathOperator{\wt}{wt}
\DeclareMathOperator{\hgt}{ht}

\DeclareMathOperator{\hs}{hs}
\DeclareMathOperator{\vs}{vs}
\DeclareMathOperator{\br}{br}
\DeclareMathOperator{\sk}{sk}
\newcommand{\set}[1]{\{#1\}}

\author{Matja\v{z} Konvalinka}
\address[Konvalinka]{
	Department of Mathematics\\
	University of Ljubljana\\
	Jadranska 21\\
	1000 Ljubljana\\
	Slovenia
        }
\email{matjaz.konvalinka@gmail.com}
\urladdr{http://www.fmf.uni-lj.si/{\small$\sim$}konvalinka}
\thanks{Matja\v z Konvalinka was partially supported by Research Programs P1-0294 and P1-0297 of the Slovenian Research Agency}

\author{Aaron Lauve}
\address[Lauve]{
	Department of Mathematics and Statistics\\
	Loyola University Chicago\\
	1032 W. Sheridan Road\\
	Chicago, IL\, 60660\\
        USA/?
        }
\email{lauve@math.luc.edu}
\urladdr{http://www.math.luc.edu/{\small$\sim$}lauve}
\thanks{Aaron Lauve was supported in part by NSA grant \#H98230-10-1-0362.}

\title{Skew Pieri Rules for Hall--Littlewood Functions}


\begin{document}

\begin{abstract}
We produce skew Pieri Rules for Hall--Littlewood functions in the spirit of Assaf and McNamara \cite{AM}. The first two were conjectured by the first author \cite{K}. The key ingredients in the proofs are a $q$-binomial identity for skew partitions and a Hopf algebraic identity that expands products of skew elements in terms of the coproduct and the antipode. 
\end{abstract}

\maketitle

Let $\sym[t]$ denote the ring of symmetric functions over $\bQ(t)$, and let $\{s_\lambda \}$ and $\{P_\lambda(t) \}$ denote its bases of Schur functions and Hall--Littlewood functions, respectively, indexed by partitions $\lambda$. The Schur functions (which are actually defined over $\bZ$) lead a rich life---making appearances in combinatorics, representation theory, and Schubert calculus, among other places. See \cite{Fu,Mac} for details. The Hall--Littlewood functions are nearly as ubiquitous (having as a salient feature that $P_\lambda(t) \to s_\lambda$ under the specialization $t\to0$). See \cite{LLT} and the references therein for their place in the literature.

\medskip

A classical problem is to determine cancellation-free formulas for multiplication in these bases,
\[
	s_\lambda \, s_\mu \ = \ \sum_\nu c_{\lambda,\mu}^{\,\,\nu} \, s_\nu
\qand
	P_\lambda \, P_\mu \ = \ \sum_\nu f_{\lambda,\mu}^{\,\nu}(t) \, P_\nu .
\]
The first problem was only given a complete solution in the latter half of the 20th century, 
while the second problem remains open. Special cases of the problem, known as \emph{Pieri rules,} have been understood for quite a bit longer. 
\medskip

The Pieri rules for Schur functions \cite[Ch. I, (5.16) and (5.17)]{Mac} take the form
\begin{equation} \label{e: cpr}
 s_\lambda \, s_{1^r} =s_\lambda \, e_r = \sum_{\lambda^+} s_{\lambda^+}\,,
\end{equation}
with the sum over partitions $\lambda^+$ for which $\lambda^+/\lambda$ is a vertical strip of size $r$, and 
\begin{equation} \label{e: pr}
 s_\lambda \, s_r = \sum_{\lambda^+} s_{\lambda^+}\,,
\end{equation}
with the sum over partitions $\lambda^+$ for which $\lambda^+/\lambda$ is a horizontal strip of size $r$.
(See Section \ref{s: prelim} for the definitions of vertical- and horizontal strip.)

\medskip

The Pieri rules for Hall--Littlewood functions \cite[Ch. III, (3.2) and (5.7)]{Mac} state that
\begin{equation} \label{e: vs}
 P_\lambda \, P_{1^r} = P_\lambda \, e_r = \sum_{|\lambda^+/\lambda| = r} \vs_{\lambda^+/\lambda}(t) P_{\lambda^+}
\end{equation}
and
\begin{equation} \label{e: hs}
 P_\lambda \, q_r = \sum_{|\lambda^+/\lambda| = r} \hs_{\lambda^+/\lambda}(t) P_{\lambda^+} \,,
\end{equation}
with the sums again running over vertical strips and horizontal strips, respectively. Here $q_r$ denotes $(1-t)P_{r}$ for $r > 0$ with $q_0 = P_0 = 1$, and $\vs_{\lambda/\mu}(t)$, $\hs_{\lambda/\mu}(t)$ are certain polynomials in $t$. (See Section \ref{s: prelim} for their definitions, as well as those of $\sk_{\lambda/\mu}(t)$ and $\br_{\lambda/\mu}(t)$ appearing below.)

\medskip

In many respects (beyond the obvious similarity of \eqref{e: pr} and \eqref{e: hs}), the $q_r$ play the same role for Hall--Littlewood functions that the $s_r$ play for Schur functions. Still, one might ask for a link between the two theories. The following generalization of \eqref{e: pr}, which seems to be missing from the literature, is our first result (Section \ref{s: prelim}).

\begin{theorem} \label{thm: P.s}
For a partition $\lambda$ and $r\geq0$, we have
\begin{equation} \label{e: P.s}
	P_{\lambda} \, s_r = 
	\sum_{\lambda^+} \sk_{\lambda^+/\lambda}(t) P_{\lambda^+}\,,
\end{equation}
with the sum over partitions $\lambda^+ \supseteq \lambda$ for which $|\lambda^+/\lambda| = r$.
\end{theorem}

\medskip

The main focus of this article is on the generalizations of Hall--Littlewood functions to skew shapes ${\lambda/\mu}$. Our specific question about skew Hall--Littlewood functions is best introduced via the recent answer for skew Schur functions $s_{\lambda/\mu}$. In \cite{AM}, Assaf and McNamara give a \demph{skew Pieri rule} for Schur functions. They prove (bijectively) the following generalization of \eqref{e: pr}:
\begin{equation} \label{e: spr}
 s_{\lambda/\mu} \, s_r = \sum_{\lambda^+,\,\mu^-} (-1)^{|\mu/\mu^-|} s_{\lambda^+/\mu^-}\,,
\end{equation}
with the sum over pairs $(\lambda^+,\mu^-)$ of partitions such that $\lambda^+/\lambda$ is a horizontal strip, $\mu/\mu^-$ is a vertical strip, and $|\lambda^+/\lambda|+|\mu/\mu^-| = r$. This elegant gluing-together of an $s_r$-type Pieri rule for the outer rim of $\lambda/\mu$ with an $e_{r}$-type Pieri rule for the inner rim of $\lambda/\mu$ demanded further exploration. 

\medskip

Before we survey the literature that followed the Assaf--McNamara result, we call attention to some work that preceded it. The skew Schur functions do not form a basis; so, from a strictly ring theoretic perspective (or representation theoretic, or geometric), it is more natural to ask how the product in \eqref{e: spr} expands in terms of Schur functions. This answer, and vast generalizations of it, was provided by Zelevinsky in \cite{Zel:1981}. In fact, \eqref{e: spr} provides such an answer as well, since 
\[
	s_{\lambda^+/\mu^-} = \sum_\nu c_{\mu^-,\nu}^{\lambda^+} \, s_\nu
\]
and the coefficients $c_{\mu^-,\nu}^{\,\lambda^+}$ are well-understood, but the resulting formula has an enormous amount of cancellation, while Zelevinsky's is cancellation free. It is an open problem to find a representation theoretic (or geometric) explanation of \eqref{e: spr}. 

\begin{remark}
As an example of the type of explanation we mean, recall Zelevinsky's realization \cite{Zel:1987} of the classical Jacobi--Trudi formula for $s_\lambda$ ($\lambda \vdash n$) from the resolution of a well-chosen polynomial representation of $\mathrm{GL}_n$. See also \cite{Akin,Doty}. 
\end{remark}

Returning to the literature that followed \cite{AM}, Lam, Sottile, and the second author \cite{LLS} found a Hopf algebraic explanation for \eqref{e: spr} that readily extended to many other settings. A skew Pieri rule for $k$-Schur functions was given, for instance, as well one for (noncommutative) ribbon Schur functions. Within the setting of Schur functions, it provided an easy extension of \eqref{e: spr} to products of arbitrary skew Schur functions---a formula first conjectured by Assaf and McNamara in \cite{AM}. (The results of this paper use the same Hopf machinery. For the non-experts, we reprise most of details and background in Section \ref{s: Hopf}.) 

\medskip

Around the same time, the first author \cite{K} was motivated to give a skew Murnaghan-Nakayama rule in the spirit of Assaf and McNamara. Along the way, he gives a bijective proof of the conjugate form of \eqref{e: spr} (only proven in \cite{AM} using the automorphism $\omega$) and a \emph{quantum} skew Murnaghan-Nakayama rule that takes the following form.
\begin{equation} \label{e: sqmnr}
 s_{\lambda/\mu} \, q_r = \sum_{\lambda^+,\mu^-} (-1)^{|\mu/\mu^-|} \br_{\lambda^+/\lambda}(t) \br_{(\mu/\mu^-)^c}(t) s_{\lambda^+/\mu^-}\,,
\end{equation}
with the sum over pairs $(\lambda^+,\mu^-)$ of partitions such that $\lambda^+/\lambda$ and $\mu/\mu^-$ are broken ribbons and $|\lambda^+/\lambda|+|\mu/\mu^-| = r$. Note that since $P_r(0) = s_r$, we recover the skew Pieri rule for $t = 0$. Also, since $P_r(1) = p_r$ (the $r$-th power sum symmetric function), we recover the skew Murnaghan-Nakayama rule \cite{AM2} if we divide the formula by $1-t$ and let $t \to 1$. 
This formula, like that in Theorem \ref{thm: P.s}, may be viewed as a link between the two theories of Schur and Hall--Littlewood functions. One is tempted to ask for other examples of mixing, e.g., swapping the rolls of Schur and Hall--Littlewood functions in \eqref{e: sqmnr}. Two such examples were found (conjecturally) in \cite{K}. Their proofs, and a generalization of \eqref{e: spr} to the Hall--Littlewood setting, are the main results of this paper.

\begin{theorem} \label{thm: e-Pieri}
For partitions $\lambda,\mu$, $\mu\subseteq \lambda$, and $r\geq0$, we have
\[
	P_{\lambda/\mu} \, s_{1^r} = P_{\lambda/\mu} \, e_r = P_{\lambda/\mu} \, P_{1^r} = 
	\sum_{\lambda^+,\mu^-} {(-1)^{|\mu/\mu^-|}} \vs_{\lambda^+/\lambda}(t) \sk_{\mu/\mu^-}(t)\, P_{\lambda^+/\mu^-}\,,
\]
where the sum on the right is over all $\lambda^+ \supseteq \lambda$, $\mu^-\subseteq \mu$ such that $|\lambda^+/\lambda| + |\mu/\mu^-| = r$.
\end{theorem}

\begin{theorem} \label{thm: s-Pieri}
For partitions $\lambda,\mu$, $\mu\subseteq \lambda$, and $r\geq0$, we have
\[
	P_{\lambda/\mu} \, s_r = 
	\sum_{\lambda^+,\mu^-} {(-1)^{|\mu/\mu^-|}} \sk_{\lambda^+/\lambda}(t) \vs_{\mu/\mu^-}(t)\, P_{\lambda^+/\mu^-}\,,
\]
where the sum on the right is over all $\lambda^+ \supseteq \lambda$, $\mu^-\subseteq \mu$ such that $|\lambda^+/\lambda| + |\mu/\mu^-| = r$.
\end{theorem}

Note that putting $\mu=\emptyset$ above recovers Theorem \ref{thm: P.s}. 
(We offer two proofs of Theorem \ref{thm: s-Pieri}; one that rests on Theorem \ref{thm: P.s} and one that does not.)

\begin{theorem} \label{thm: q-Pieri}
For partitions $\lambda,\mu$, $\mu\subseteq \lambda$, and $r\geq0$, we have
\[
	P_{\lambda/\mu} \, q_r = 	\sum_{\lambda^+,\mu^-,\nu}  (-1)^{|\mu/\mu^-|} (-t)^{|\nu/\mu^-|} \hs_{\lambda^+/\lambda}(t)\vs_{\mu/\nu}(t)\, \sk_{\nu/\mu^-}(t) \, P_{\lambda^+/\mu^-}\,,
\]
where the sum on the right is over all $\lambda^+ \supseteq \lambda$, $\mu^-\subseteq \nu \subseteq \mu$ such that $|\lambda^+/\lambda| + |\mu/\mu^-| = r$.
\end{theorem}

\begin{remark}
We reiterate that the skew elements do not form a basis for $\sym[t]$, so the expansions announced in Theorems \ref{thm: e-Pieri}--\ref{thm: q-Pieri} are by no means unique. However, if we demand that the expansions be over partitions $\lambda^+\supseteq \lambda$ and $\mu^-\subseteq \mu$, and that the coefficients factor nicely as products of polynomials $a_{\lambda^+/\lambda}(t)$ (independent of $\mu$) and $b_{\mu/\mu^-}(t)$ (independent of $\lambda$), then they are in fact unique (up to scalar). We make this remark precise in Theorem \ref{thm: unique} in Section \ref{s: skew proofs}. 
\end{remark}

This paper is organized as follows. In Section \ref{s: prelim}, we prove some polynomial identities involving $\hs$, $\vs$ and $\sk$, prove Theorem \ref{thm: P.s}, and find $\omega(q_r)$. In Section \ref{s: Hopf}, we introduce our main tool, Hopf algebras. We conclude in Section \ref{s: skew proofs} with the proofs of our main theorems.

\section{Combinatorial Preliminaries}\label{s: prelim}

\subsection{Notation, and a key lemma} \label{s: key lemma}

The conjugate partition
of $\lambda$ is denoted $\lambda^c$. We write $\m{i}{\lambda}$ for the number of parts of $\lambda$ equal to $i$. The $q$-binomial coefficient is defined by
$$\qbinom{a}{b}{q} = \frac{(1-q^a)(1-q^{a-1})\cdots(1-q^{a-b+1})}{(1-q^b)(1-q^{b-1})\cdots(1-q)}$$
and is a polynomial in $q$ that gives $\binom a b$ when $q = 1$. For a partition $\lambda$, define $n(\lambda) = \sum_i (i-1)\lambda_i = \sum_i \binom{\lambda_i^c}2$.
\medskip

Given two partitions $\lambda$ and $\mu$, we say $\mu\subseteq \lambda$ if $\lambda_i \geq \mu_i$ for all $i\geq1$, in which case we may consider the pair as a \demph{skew shape} $\lambda/\mu$. We write $[\lambda/\mu]$ for the cells $\{(i,j) \colon 1\leq i\leq \ell(\lambda),\, \mu_i < j \leq \lambda_i\}$. 
We say that $\lambda/\mu$ is a \demph{horizontal strip} (respectively \demph{vertical strip}) if $[\lambda/\mu]$ contains no $2 \times 1$ (respectively $1 \times 2$) block, equivalently, if $\lambda_i^c \leq \mu_i^c + 1$ (respectively $\lambda_i \leq \mu_i+1$) for all $i$. We say that $\lambda/\mu$ is a \demph{ribbon} if $[\lambda/\mu]$ is connected and if it contains no $2 \times 2$ block, and that $\lambda/\mu$ is a \demph{broken ribbon} if $[\lambda/\mu]$ contains no $2 \times 2$ block, equivalently, if $\lambda_i \leq \mu_{i-1}+1$ for $i \geq 2$. The Young diagram of a broken ribbon is a disjoint union of $\rib(\lambda/\mu)$ number of ribbons. The \demph{height} $\hgt(\lambda/\mu)$ (respectively \demph{width} $\wt(\lambda/\mu)$) of a ribbon is the number of non-empty rows (respectively columns) of $[\lambda/\mu]$, minus $1$. The height (respectively width) of a broken ribbon is the sum of heights (respectively widths) of the components.
\medskip

Let us define some polynomials. For a horizontal strip $\lambda/\mu$, define 
$$
	\hs_{\lambda/\mu}(t) = \prod_{\begin{array}{c} {\scriptstyle \lambda^c_j = \mu^c_j+1} \\ {\scriptstyle \lambda^c_{j+1} = \mu^c_{j+1}} \end{array}} (1-t^{m_j(\lambda)}).
$$
If $\lambda/\mu$ is not a horizontal strip, define $\hs_{\lambda/\mu}(t) = 0$. 
For a vertical strip $\lambda/\mu$, define
$$
	\vs_{\lambda/\mu}(t) = \prod_{j \geq 1} \begin{bmatrix} \lambda^c_j - \lambda^c_{j+1} \\ \lambda^c_j - \mu^c_j \end{bmatrix}_t.
$$
If $\lambda/\mu$ is not a vertical strip, define $\vs_{\lambda/\mu}(t) = 0$. 
For a broken ribbon $\lambda/\mu$, define 
$$
	\br_{\lambda/\mu}(t) = (-t)^{\hgt(\lambda/\mu)} (1-t)^{\rib(\lambda/\mu)}.
$$
If $\lambda/\mu$ is not a broken ribbon, define $\br_{\lambda/\mu}(t) = 0$. For any skew shape $\lambda/\mu$, define
$$
	\sk_{\lambda/\mu}(t) = t^{\sum_j \binom{\lambda^c_j - \mu^c_j}{2}} \prod_{j\geq1} \begin{bmatrix}  \lambda^c_j - \mu ^c_{j+1} \\ m_j(\mu) \end{bmatrix}_t.
$$
\medskip

Next, recall the \emph{$q$-binomial theorem}. For all $n,k\geq0$, we have
\begin{equation} \label{qbt}
 \prod_{i=0}^{n-1}(t+q^i) = \sum_{k=0}^n q^{\binom {n-k} 2} \qbinom n k q t^k.
\end{equation}
This may be proven by induction from the standard identity $\qbinom n k q = q^k \qbinom {n-1} k q + \qbinom{n-1} {k-1} q$.

\begin{lemma} \label{l: hs}
 For fixed partitions $\lambda,\mu$ satisfying $\mu \subseteq \lambda$, we have
 $$\sum_\nu (-t)^{|\lambda/\nu|} \vs_{\lambda/\nu}(t) \sk_{\nu/\mu}(t) = \hs_{\lambda/\mu}(t),$$
 with the sum over all $\nu$, $\mu \subseteq \nu \subseteq \lambda$, for which $\lambda/\nu$ is a vertical strip.
\end{lemma}

\begin{proof}
 Let $a_j = \lambda^c_j - \max(\mu^c_j,\lambda^c_{j+1}) \geq 0$. A partition $\nu$, $\mu \subseteq \nu \subseteq \lambda$, for which $\lambda/\nu$ is a vertical strip is obtained by choosing $k_j$, $0 \leq k_j \leq a_j$, and removing $k_j$ bottom cells of column $j$ in $\lambda$. See Figure \ref{fig1} for the example for $\lambda = 98886666444$ and $\mu = 77666633331$, where $a_4 = 3$, $a_6 = 2$, $a_8 = 3$, $a_9 = 1$ and $a_i = 0$ for all other $i$.
\begin{figure}[ht!]
 \begin{center}
    \includegraphics[height=2.75cm]{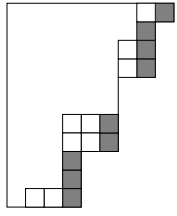}
   \caption{A partition $\nu$ ($\mu \subseteq \nu \subseteq \lambda$) for which $\lambda/\nu$ is a vertical strip within $\lambda/\mu$ is built from $\lambda$ by removing some number of the shaded cells of $[\lambda]$.
}
   \label{fig1}
 \end{center}
\end{figure}

We have $|\lambda/\nu| = \sum_j k_j$, $\nu^c_j = \lambda^c_j - k_j$. The choices of the $k_j$ are independent, which means that 
$$
 	\sum_\nu (-t)^{|\lambda/\nu|} \sk_{\nu/\mu}(t) \vs_{\lambda/\nu}(t) = 
	\sum_{k_1,k_2,\ldots} (-t)^{\sum_j k_j} t^{\sum_j \binom{\nu^c_j - \mu^c_j}{2}} \prod_j \qbinom{\nu^c_j - \mu ^c_{j+1}}{m_j(\mu)} t \prod_j \qbinom{\lambda^c_j - \lambda^c_{j+1}}{\lambda^c_j - \nu^c_j}t
$$ 
\begin{equation}\label{e: prod-sum}
	= \prod_j \sum_{k_j = 0}^{a_j} (-t)^{k_j} t^{\binom{\lambda^c_j - \mu^c_j -k_j} 2} \qbinom{\lambda^c_j-k_j-\mu^c_{j+1}}{m_j(\mu)}t \qbinom{m_j(\lambda)}{k_j}t. 
\end{equation} 
We analyze \eqref{e: prod-sum} case-by-case, showing that it reduces to $\hs_{\lambda/\mu}(t)$ when $\lambda/\mu$ is a horizontal strip and zero otherwise. Assume first that $\lambda/\mu$ is a horizontal strip. This means that $a_j \leq \lambda^c_j - \mu^c_j \leq 1$ for all $j$. 
\smallskip

\noindent\emph{Case 1: $a_j=0$.} We have $\max(\mu^c_j,\lambda^c_{j+1}) = \lambda^c_j$, so the inner sum in \eqref{e: prod-sum} is equal to 
$$
   \qbinom{\lambda^c_j-\mu^c_{j+1}}{m_j(\mu)}t= \qbinom{\lambda^c_j-\mu^c_{j+1}}{\mu^c_j-\mu^c_{j+1}}t.
$$
If $\mu^c_j = \lambda^c_j$, this is $1$, and if $\mu^c_j = \lambda^c_j - 1$ and $\lambda^c_{j+1} = \lambda^c_j$, then $\mu^c_{j+1} = \mu^c_j$ and so the expression is also $1$. 
\smallskip

\noindent\emph{Case 2: $a_j = 1$.}  This holds if and only if $\lambda^c_j = \mu^c_j + 1$, $\lambda^c_{j+1} \leq \lambda^c_j - 1$, in which case the sum in \eqref{e: prod-sum} is 
$$
	(-t)^0 t^{\binom 1 2} \qbinom{1+m_j(\mu)}{m_j(\mu)}t \qbinom{m_j(\lambda)}{0}t + (-t)^1 t^{\binom 0 2} \qbinom{m_j(\mu)}{m_j(\mu)}t \qbinom{m_j(\lambda)}{1}t  
$$ 
$$
	= 1+t+\ldots + t^{m_j(\mu)} - t\bigl(1+t+\ldots+t^{m_j(\lambda)-1}\bigr) = \left\{ \begin{array}{c@{\ \ }c@{\ \ }l} 1 - t^{m_j(\lambda)} & : & \lambda^c_j = \mu^c_j + 1, \lambda^c_{j+1} 
	= \mu^c_{j+1} \\ 1 & : & \mbox{otherwise} \end{array} \right.\!\!.
$$
Indeed, $\lambda^c_j = \mu_j^c+1$ and $\lambda^c_{j+1} = \mu_{j+1}^c+1$ imply $m_j(\mu)=m_j(\lambda)$, while $\lambda^c_j = \mu_j^c+1$ and $\lambda^c_{j+1} = \mu_{j+1}^c$ imply $\lambda^c_{j+1} \leq \mu^c_j = \lambda^c_j - 1$ and $m_j(\mu)=m_j(\lambda)-1$. Thus \eqref{e: prod-sum} equals $\hs_{\lambda/\mu}(t)$ whenever $\lambda/\mu$ is a horizontal strip.
\smallskip

Now assume that $\lambda/\mu$ is not a horizontal strip. Let $j$ be the largest index for which $\lambda^c_j - \mu^c_j \geq 2$. Let us investigate two cases, when $\lambda^c_{j+1} > \mu^c_j$ and when $\lambda^c_{j+1} \leq \mu^c_j$.
\smallskip

\noindent\emph{Case 1: $\lambda^c_{j+1} > \mu^c_j$.} We must have $\lambda^c_{j+1} = \mu^c_j + 1$ and $\mu^c_{j+1} = \mu^c_j$, for otherwise $\lambda^c_{j+1} - \mu^c_{j+1} = (\lambda^c_{j+1} - \mu^c_j) + (\mu^c_j - \mu^c_{j+1}) \geq 2$, which contradicts the maximality of $j$. So $a_j = m_j(\lambda)$, $\lambda^c_j - \mu^c_j =\lambda^c_j - \mu^c_{j+1} = m_j(\lambda)+1$, $m_j(\mu) = 0$, $m_j(\lambda) \geq 1$ and
 \begin{align*}
 \phantom{=}& \sum_{k_j = 0}^{a_j} (-t)^{k_j} t^{\binom{\lambda^c_j - \mu^c_j -k_j} 2} \qbinom{\lambda^c_j-k_j-\mu^c_{j+1}}{m_j(\mu)}t \qbinom{m_j(\lambda)}{k_j}t = \sum_{k_j = 0}^{m_j(\lambda)} (-t)^{k_j} t^{\binom{m_j(\lambda)+1 -k_j} 2} \qbinom{m_j(\lambda)}{k_j}t \\
 = &\sum_{k_j = 0}^{m_j(\lambda)} (-t)^{k_j} t^{\binom{m_j(\lambda)-k_j} 2 + m_j (\lambda)- k_j} \qbinom{m_j(\lambda)}{k_j}t = t^{m_j(\lambda)} \sum_{k_j=0}^{m_j(\lambda)} (-1)^{k_j} t^{\binom{m_j(\lambda)-k_j} 2} \qbinom{m_j(\lambda)}{k_j}t.
 \end{align*}
Using \eqref{qbt} with $n = m_j(\lambda)$, $t = -1$ and $q = t$, the above simplifies to
$$
   t^{m_j(\lambda)} \prod_{i=0}^{m_j(\lambda)-1} (-1+t^i) = 0.
$$
\smallskip

\noindent\emph{Case 2: $\lambda^c_{j+1} \leq \mu^c_j$.} We consider two further options. If $\mu^c_{j+1} = \lambda^c_{j+1}$, then $a_j = \lambda^c_j - \mu^c_j = m_j(\lambda) - m_j(\mu) \geq 2$ and
 \begin{align*}
 \phantom{=}&\sum_{k_j = 0}^{a_j} (-t)^{k_j} t^{\binom{\lambda^c_j - \mu^c_j -k_j} 2} \qbinom{\lambda^c_j-k_j-\mu^c_{j+1}}{m_j(\mu)}t \qbinom{m_j(\lambda)}{k_j}t  \\
=&\sum_{k_j = 0}^{m_j(\lambda)-m_j(\mu)} (-t)^{k_j} t^{\binom{m_j(\lambda)-m_j(\mu) -k_j} 2} \qbinom{m_j(\lambda)-k_j}{m_j(\mu)}t \qbinom{m_j(\lambda)}{k_j}t \\
 =& \sum_{k_j = 0}^{m_j(\lambda)-m_j(\mu)} (-t)^{k_j} t^{\binom{m_j(\lambda)-m_j(\mu) -k_j} 2} \qbinom{m_j(\lambda)-m_j(\mu)}{k_j}t \qbinom{m_j(\lambda)}{m_j(\mu)}t.
 \end{align*}
 If we use \eqref{qbt} with $n = m_j(\lambda)-m_j(\mu)$, $t = -t$ and $q = t$, we get
 $$
    \qbinom{m_j(\lambda)}{m_j(\mu)}t \prod_{i=0}^{m_j(\lambda)-m_j(\mu) - 1} (-t + t^i) = 0.
 $$
 On the other hand, if $\mu^c_{j+1} = \lambda^c_{j+1} - 1$, then $a_j = \lambda^c_j - \mu^c_j = m_j(\lambda) - m_j(\mu) + 1\geq 2$ and
\begin{flalign*}
    \phantom{=}&\sum_{k_j = 0}^{a_j} (-t)^{k_j} t^{\binom{\lambda^c_j - \mu^c_j -k_j} 2} \qbinom{\lambda^c_j-k_j-\mu^c_{j+1}}{m_j(\mu)}t \qbinom{m_j(\lambda)}{k_j}t&\\
 &=\sum_{k_j = 0}^{m_j(\lambda) - m_j(\mu) + 1} (-t)^{k_j} t^{\binom{m_j(\lambda) - m_j(\mu) + 1 -k_j} 2} \qbinom{m_j(\lambda)+1-k_j}{m_j(\mu)}t \qbinom{m_j(\lambda)}{k_j}t&\\
 &=\!\!\!\!\!\sum_{k_j = 0}^{m_j(\lambda) - m_j(\mu) + 1} \!\!\!\!\!(-t)^{k_j} t^{\binom{m_j(\lambda) - m_j(\mu) + 1 -k_j} 2} \frac{1 - t^{m_j(\lambda)+1-k_j}}{1-t^{m_j(\lambda)-m_j(\mu) + 1}} \qbinom{m_j(\lambda)-m_j(\mu)+1}{k_j}t \qbinom{m_j(\lambda)}{m_j(\mu)}t&
\end{flalign*}
\begin{flalign*}
&=\!\frac{1}{1-t^{m_j(\lambda)-m_j(\mu) + 1}} \qbinom{m_j(\lambda)}{m_j(\mu)}t \!\!\left( \sum_{k_j = 0}^{m_j(\lambda) - m_j(\mu) + 1} \!\!\!\!\!\!\! (-t)^{k_j} t^{\binom{m_j(\lambda) - m_j(\mu) + 1 -k_j} 2} \!\qbinom{m_j(\lambda)-m_j(\mu)+1}{k_j}t\right.&\\
&-\left. \sum_{k_j = 0}^{m_j(\lambda) - m_j(\mu) + 1} (-1)^{k_j} t^{\binom{m_j(\lambda) - m_j(\mu) + 1 -k_j} 2} t^{m_j(\lambda)+1}\qbinom{m_j(\lambda)-m_j(\mu)+1}{k_j}t \right).&
 \end{flalign*} 
We prove that the first (respectively, second) sum is $0$ by substituting $n = m_j(\lambda)-m_j(\mu) + 1$, $t = -t$ (respectively, $t = -1$) and $q = t$ in \eqref{qbt}. This finishes the proof of the lemma.
\end{proof}

\subsection{Elementary Hall--Littlewood identities} \label{s: id}

We give two applications of Lemma \ref{l: hs}, then prove some elementary properties on Hall--Littlewood functions that will be useful in Section \ref{s: skew proofs}. The first application is a formula for the product of a Hall--Littlewood polynomial with the Schur function $s_r$. 

\begin{proof}[Proof of Theorem \ref{thm: P.s}]
 The proof is by induction on $r$. For $r = 0$, there is nothing to prove. For $r > 0$, we use the formula
 \begin{equation} \label{qr}
  q_r = \sum_{k=0}^r (-t)^k s_{r-k}e_k,
 \end{equation}
which is proven as follows. It is well-known and easy to prove (see e.g.\ \cite[Exercise 7.11]{St}) that
$$
   P_r = \sum_{\tau\, \vdash\, n} (1-t)^{\ell(\tau)-1} m_\tau = \sum_{k=0}^{r-1} (-t)^k s_{r-k,1^k}\,.
$$
The conjugate Pieri rule then gives \eqref{qr}, for
$$
   \sum_{k=0}^r (-t)^k s_{r-k} e_k = s_r + \sum_{k=1}^{r-1} (-t)^k (s_{r-k,1^k}+s_{r-k+1,1^{k-1}}) + (-t)^r s_{1^r} = q_r\,.
$$
For $|\lambda^+/\lambda| = r$, the coefficient of $P_{\lambda^+}$ in
$$
   P_\lambda \, s_r = P_\lambda \left( q_r - \sum_{k=1}^r (-t)^k s_{r-k} e_k\right)
$$
reduces by induction, \eqref{e: vs} and \eqref{e: hs} to
$$
   \hs_{\lambda^+/\lambda}(t) - \sum (-t)^{|\lambda^+/\nu|} \sk_{\nu/\lambda}(t) \vs_{\lambda^+/\nu}(t),
$$
with the sum over all $\nu$, $\lambda \subseteq \nu \subseteq \lambda^+$, for which $\lambda^+/\nu$ is a vertical strip of size at least $1$. By Lemma \ref{l: hs}, this is equal to $\sk_{\lambda^+/\lambda}(t)$.
\end{proof}

Recall that $f^\lambda_{\mu,\tau}(t)$ is the (polynomial) coefficient of $P_\lambda$ in $P_\mu P_\tau$.

\begin{corollary}\label{cor: P.s}
The structure constants $f_{\mu,\tau}^\lambda(t)$ satisfy 
$\displaystyle
	\sum_\tau t^{n(\tau)} f^\lambda_{\mu,\tau}(t) = \sk_{\lambda/\mu}(t).
$
\end{corollary}

\begin{proof}This follows from $s_r = \sum_{\tau \vdash r} t^{n(\tau)} P_\tau$, which is (2) in \cite[page 219]{Mac} and also Theorem \ref{thm: P.s} for $\lambda = \emptyset$.
\end{proof}

The second application of Lemma \ref{l: hs} is the following generalization of Example 1 of \cite[\S III.3, Example 1]{Mac}. 

\begin{theorem} \label{thm: y}
  For every $\lambda,\mu$, we have
 \begin{equation} \label{eq1}
  \sum_{\nu} \vs_{\lambda/\nu}(t) \sk_{\nu/\mu}(t) y^{|\lambda/\nu|} = \sum_\sigma t^{n(\sigma)-\binom{\ell(\sigma)}2} f_{\sigma \mu}^\lambda(t)\prod_{j=1}^{\ell(\sigma)} (y+t^{j-1}).
 \end{equation}
 Equivalently, for all $m$,
 \begin{equation} \label{eq2}
  \sum_{\nu \colon |\lambda/\nu|=m} \vs_{\lambda/\nu}(t) \sk_{\nu/\mu}(t) = \sum_\sigma t^{n(\sigma)-\binom m 2} f_{\sigma \mu}^\lambda(t) \qbinom{\ell(\sigma)} m {t^{-1}}.
 \end{equation}
\end{theorem}

\begin{proof}
 Let us evaluate 
$
	P_\mu \, s_r \left( \sum_m e_m \, y^m \right) 
$ 
in two different ways. On the one hand, 
$$
	P_\mu \, s_r \left( \sum_m e_m \, y^m \right) = \left( \sum_\nu \sk_{\nu/\mu}(t) P_\nu \right)  \!\left( \sum_m e_m \, y^m \right) 
	= \sum_{\nu,\lambda} \sk_{\nu/\mu}(t) \vs_{\lambda/\nu}(t) P_\lambda \, y^{|\lambda/\nu|}. 
$$ 
On the other hand, using Example 1 on page 218 of \cite{Mac},
$$
	P_\mu \, s_r \left( \sum_m e_m y^m \right) = P_\mu \sum_\sigma t^{n(\sigma)} P_\sigma \prod_{j=1}^{\ell(\sigma)} (1+t^{1-j}y) 
	= \sum_{\sigma,\lambda} t^{n(\sigma)-\binom{\ell(\sigma)}2} f_{\sigma\mu}^\lambda(t) P_\lambda \prod_{j=1}^{\ell(\sigma)}(y+t^{j-1}).
$$
Now \eqref{eq1} follows by taking the coefficient of $P_\lambda$ in both expressions. For \eqref{eq2}, we use the $q$-binomial theorem \eqref{qbt} and
$$
	\qbinom n k {t^{-1}} = t^{\binom k 2 + \binom {n-k}2 - \binom n 2} \qbinom n k t. 
$$ 
\end{proof}

\begin{remark}
 The theorem is indeed a generalization of \cite[\S III.3, Example 1]{Mac}. For $\mu = \emptyset$, $\sk_{\nu/\mu}(t) = t^{n(\nu)}$, and the right-hand side of \eqref{eq2} is non-zero only for $\sigma = \lambda$, so the last equation on page 218 (\emph{loc.\ cit.}) follows. It also generalizes Lemma \ref{l: hs}: for $y = -t$, the right-hand side of \eqref{eq1} is non-zero if and only if $\ell(\sigma) = 1$, and is therefore equal to $\hs_{\lambda/\mu}(t)$.
\end{remark}

We finish the section with two more lemmas.

\begin{lemma}  Given $r > k\geq0$, we have 
$$
 	s_{r-k,1^k} = \sum_{\lambda \colon \ell(\lambda) \geq k + 1} t^{\binom{\ell(\lambda)-k}2 + \sum_{i=2}^{\lambda_1} \binom{\lambda_i^c}2} \qbinom{\ell(\lambda)-1} k t P_\lambda.
$$
\end{lemma}

\begin{proof}
 The lemma follows from a formula due to Lascoux and Sch\" utzenberger. See \cite[Ch.~III,~(6.5)]{Mac}. In that terminology, we have to evaluate $K_{(r-k,1^k),\lambda}(t)$. We choose a semistandard Young tableau $T$ of shape $(r-k,1^k)$ and type $\lambda=(\lambda_1,\ldots,\lambda_\ell)$. Clearly, such tableaux are in one-to-one correspondence with $k$-subsets of the set $\set{2,\ldots,\ell}$. For such a subset $S$, write $s$ for the word with the elements of $S$ in increasing order, and write $\overline s$ for the word with the elements of $\set{2,\ldots,\ell} \setminus S$ in decreasing order. The reverse reading word of the tableau corresponding to $S$ is $\ell^{\lambda_\ell - 1} \cdots 3^{\lambda_3-1} 2^{\lambda_2-1} 1^{\lambda_1} s$. The subwords $w_2,w_3,\ldots$ are all strictly decreasing, and $w_1 = \overline s 1 s$. The charges of $w_2,w_3,\ldots$ are $\binom{\lambda_2^c}2,\binom{\lambda_3^c}2,\ldots$, while the charge of $w_1$ is $\sum_{i \notin S} (\ell - i + 1)$ (sum over $i \notin S$, $2 \leq i \leq \ell$). We have
 $$\sum_{S \subseteq \set{2,\ldots,\ell+1},|S|=k} \!\!\!\!\!\!\!t^{\sum_{i \notin S} (\ell + 1 - i + 1)} = \!\!\!\!\!\!\!\sum_{S \subseteq \set{2,\ldots,\ell},|S|=k-1}\!\!\!\!\!\!\! t^{\sum_{i \notin S} (\ell + 1 - i + 1)} + \!\!\!\!\!\!\!\sum_{S \subseteq \set{2,\ldots,\ell},|S|=k} \!\!\!\!\!\!\!t^{1 + \sum_{i \notin S} (\ell + 1 - i + 1)},$$
 and the formula
 $$\sum_{S \subseteq \set{2,\ldots,\ell},|S|=k} t^{\sum_{i \notin S} (\ell - i + 1)} = t^{\binom{\ell - k}2} \qbinom{\ell - 1} k t$$
 follows by induction on $\ell$. This finishes the proof.
\end{proof}

\begin{lemma} \label{l: omega}
Let $\omega$ be the fundamental involution on $\sym[t]$ defined by $\omega(s_\lambda) = s_{\lambda^c}$. 
We have
$$
  \omega(q_r) = (-1)^{r} \sum_{\lambda \vdash r} c_\lambda(t) P_\lambda,
$$
where
$$
   c_\lambda(t) = t^{\sum_{i=2}^{\lambda_1} \binom{\lambda_i^c+1}2} \prod_{i=1}^{\ell(\lambda)} (-1+t^i).
$$
\end{lemma}
\begin{proof}
 We have
 \begin{align*}
   \omega(P_r) &= \omega \left( \sum_{k=0}^{r-1} (-t)^{r-k-1} s_{k+1,1^{r-k-1}}\right) = \sum_{k=0}^{r-1} (-t)^{r-k-1} s_{r-k,1^k} = \\
   &= \sum_{k=0}^{r-1} (-t)^{r-k-1} \left( \sum_{\ell(\lambda) \geq k+1} t^{\binom{\ell(\lambda)-k}2 + \sum_{i=2}^{\lambda_1} \binom{\lambda_i^c}2} \qbinom{\ell(\lambda)-1} k t P_\lambda \right)= \\
   &=\sum_{\lambda \vdash r} \left( \sum_{k=0}^{\ell(\lambda)-1}(-t)^{r-k-1} t^{\binom{\ell(\lambda)-k}2 + \sum_{i=2}^{\lambda_1} \binom{\lambda_i^c}2} \qbinom{\ell(\lambda)-1} k t \right) P_\lambda.
\end{align*}
 Now by the $q$-binomial theorem,
$$
   \prod_{i=2}^{\ell(\lambda)} (-1+t^i) = t^{2(\ell(\lambda)-1)} \prod_{i=0}^{\ell(\lambda)-2} (-1/t^2 + t^i) = t^{2(\ell(\lambda)-1)}  \sum_{k=0}^{\ell(\lambda)-1} t^{\binom{\ell(\lambda)-1-k}{2}} \qbinom{\ell(\lambda)-1} k t \left( - \frac 1 {t^2}\right)^k.
$$
Simple calculations now show that the coefficient of $P_\lambda$ in $\omega(q_r) = (1-t) \omega(P_r)$ is indeed $(-1)^{r} c_\lambda(t)$.
\end{proof}

\section{Hopf Perspective on Skew Elements}\label{s: Hopf}

Recall that $\sym[t]$ has another important basis $\set{Q_\lambda}$, defined by $Q_\lambda = b_\lambda(t) P_\lambda$, where $b_\lambda(t) = \prod_{i \geq 1} (1-t)(1-t^2)\cdots(1-t^{m_i(\lambda)})$. The (extended) Hall scalar product on $\sym[t]$ is uniquely defined by either of the (equivalent) conditions
\[
	\innp(P_\lambda,Q_\mu) = \delta_{\lambda \mu} 
	\qor
	\innp(p_\lambda,p_\mu) = {z_\mu(t)}\,\delta_{\lambda \mu} \,,
\]
where, taking $\mu=(\mu_1,\mu_2,\dotsc,\mu_r)=\langle 1^{a_1},2^{a_2},\dotsb, k^{a_k} \rangle$,
\[
	z_\mu(t) \ = \ z_\mu\cdot \prod_{j=1}^r(1-t^{\mu_j})^{-1} 
		\ = \  \prod_{i=1}^k \left(i^{a_i}a_i!\right) \prod_{j=1}^r(1-t^{\mu_j})^{-1} \,.
\]
See \cite[\S III.4]{Mac}. The skew Hall--Littlewood function $P_{\lambda/\mu}$ is defined \cite[Ch. III, (5.1$'$)]{Mac} as the unique function satisfying
\begin{equation} \label{e: skew-via-duality}
	\innp(P_{\lambda/\mu},Q_\nu) = \innp(P_\lambda, Q_\nu\,Q_\mu)
\end{equation}
for all $Q_\nu \in \sym[t]$. (Likewise for $Q_{\lambda/\mu}$.) If we choose to read $P_{\lambda/\mu}$ as, ``$Q_\mu$ skews $P_\lambda$,'' then we allow ourselves access to the machinery of Hopf algebra actions on their duals. We introduce the basics in Subsection \ref{s: Hopf prelim} and return to $\sym[t]$ and Hall--Littlewood functions in Subsection \ref{s: HL setting}. 

\subsection{Hopf preliminaries}\label{s: Hopf prelim}
Let $H=\bigoplus_n H_n$ be a graded algebra over a field $\field$. Recall that $H$ is a Hopf algebra if there are algebra maps $\Delta\colon H \to H\otimes H$, $\varepsilon\colon H\to\field$, and an algebra antimorphism $S\colon H \to H$, called the \demph{coproduct}, \demph{counit}, and \demph{antipode}, respectively, satisfying some additional compatibility conditions. See \cite{Mont}. 

Let $H^* = \bigoplus_n H_n^*$ denote the graded dual of $H$. If each $H_n$ is finite dimensional, then the pairing 
$
	\innp(\,\cdot\,,\,\cdot\,)\colon H \otimes H^* \to \field
$ 
defined by $\innp(h,a) = a(h)$ is nondegenerate. This pairing naturally endows $H^*$ with a Hopf algebra structure, with product and coproduct uniquely determined by the formulas:
\[
	\innp(h,a\cdot b) := \innp(\Delta(h),a\otimes b) 
\qand
	\innp(g\otimes h,\Delta(a)) := \innp(g\cdot h,a) 
\]
for all homogeneous $g,h\in H$ and $a,b \in H^*$. (Extend to all of $H^*$ by linearity, insisting that $\innp(H_n,H_m^*) = 0$ for $n\neq m$.)

\begin{remark}
The finite dimensionality of $H_n$ ensures that the coproduct in $H^*$ is a finite sum of functionals, $\Delta(a) = \sum_{(a)} a'\otimes a''$. Here and below we use Sweedler's notation for coproducts.
\end{remark} 

We now recall some standard actions (``$\harpoon$'') of $H$ and $H^*$ on each other. Given $h\in H$ and $a\in H^*$, put
\begin{gather}\label{e: harpoon}
	a\harpoon h := \sum_{(h)} \innp(h'',a)h'
	\qand
	h\harpoon a := \sum_{(a)} \innp(h,a'')a'.
\end{gather}
Equivalently, $\innp(g,h\harpoon a) = \innp(g\cdot h,a)$ and $\innp(a\harpoon h,b) = \innp(h,b\cdot a)$. We call these \emph{skew elements} (in $H$ and $H^*$, respectively) to keep the nomenclature consistent with that in symmetric function theory.

Our skew Pieri rules (Theorems \ref{thm: e-Pieri}, \ref{thm: s-Pieri} and \ref{thm: q-Pieri}) come from an elementary formula relating products of elements $h$ and skew elements $a\harpoon g$ in a Hopf algebra $H$:
\begin{equation}\label{e: Hopf skew rule}
	(a\harpoon g) \cdot h = 
	\sum \bigl( S(h'') \harpoon a\bigr) \harpoon (g\cdot h').
\end{equation}
See ($\ast$) in the proof of \cite[Lemma 2.1.4]{Mont} or \cite[Lemma 1]{LLS}. Before turning to the proofs of these theorems, we first recall  the Hopf structure of $\sym[t]$.

\subsection{The Hall--Littlewood setting}\label{s: HL setting} 
The ring $\sym[t]$ is generated by the one-part power sum symmetric functions $p_r$ ($r>0$), so the definitions 
\begin{gather}\label{e: p-Hopf structure}
	\Delta(p_r) := 1\otimes p_r + p_r \otimes 1, \quad
	\varepsilon(p_r) := 0, \qand
	S(p_r) := -p_r 
\end{gather}
completely determine the Hopf structure of $\sym[t]$. 

\begin{proposition} \label{p: esq}
For $r>0$,
\begin{align*}
 \Delta(e_r) &= \sum_{k = 0}^r e_k \otimes e_{r-k} & 
 \Delta(s_r) &= \sum_{k = 0}^r s_k \otimes s_{r-k} &  \Delta(q_r) &= \sum_{k = 0}^r q_k \otimes q_{r-k} & \\
S(e_r) &= (-1)^r s_r & S(s_r) &= (-1)^r e_r  & S(q_r) &= \sum_{\lambda \vdash r} c_\lambda P_\lambda.
\end{align*}
where $c_\lambda$ is given by Lemma \ref{l: omega}. 
\end{proposition}

\begin{proof} 
Equalities for $e_r$ and $s_r$ are elementary consequences of \eqref{e: p-Hopf structure} and may be found in \cite[\S I.5, Example 25]{Mac}. The coproduct formula for $q_r$ is (2) in \cite[\S III.5, Example 8]{Mac}. The antipode formula for $q_r$ is identical to Lemma \ref{l: omega}, as the fundamental morphism $\omega$ and the antipode $S$ are related by $S(h) = (-1)^r\omega(h)$ on homogeneous elements $h$ of degree $r$.
\end{proof}

It happens that $\sym[t]$ is self-dual as a Hopf algebra. This may be deduced from Example 8 in \cite[\S III.5]{Mac}, but we illustrate it here in the power sum basis for the reader not versed in Hopf formalism.

\begin{lemma}\label{l: self-dual}
The Hopf algebra $\sym[t]$ is self-dual with the extended Hall scalar product.
\end{lemma}

\begin{proof}
Write $p^*_\lambda$ for $z_\lambda(t)^{-1} p_\lambda$. It is sufficient to check that 
\[
   \innp(p_\lambda,p^*_\mu \cdot p^*_\nu) = \innp(\Delta(p_\lambda),p^*_\mu \otimes p^*_\nu)
   \qand
   \innp(p_\mu\otimes p_\nu,\Delta(p^*_\lambda{)}) = \innp(p_\mu \cdot p_\nu,p^*_\lambda)
\]
for all partitions $\lambda,\mu$, and $\nu$. 

\smallskip\noindent\emph{Products and coproducts in the power sum basis.}\ 
Given partitions $\lambda = \langle 1^{m_1},2^{m_2},\dotsb\rangle$ and $\mu = \langle 1^{n_1},2^{n_2},\dotsb \rangle$, we write $\lambda\cup\mu$ for the partition $\langle 1^{m_1+n_1},2^{m_2+n_2},\dotsb \rangle$. Also, we write $\mu \leq \lambda$ if $n_i\leq m_i$ for all $i\geq1$. In this case, we define 
\[
  \binom{\lambda}{\mu} = \prod_{i\geq1} \binom{m_i}{n_i} ,
\]
and otherwise define $\binom{\lambda}{\mu} = 0$. Since the power sum basis is multiplicative ($p_\lambda = \prod_{i\geq1} p_{\lambda_i}$), we have
$
    p_\mu \cdot p_\nu = p_{\mu\cup\nu}.
$ 
Since $\Delta$ is an algebra map, 
the first formula in \eqref{e: p-Hopf structure} gives
\[
    \Delta(p_\lambda) = \sum_{\substack{\mu \leq \lambda \\ \mu\cup\nu = \lambda}}
    \binom{\lambda}{\mu} p_\mu \otimes p_\nu \,.
\]

\smallskip\noindent\emph{Products and coproducts in dual basis.}\ 
It is easy to see that
\begin{gather}\label{e: zee}
	z_\lambda(t)^{-1} \cdot \binom{\lambda}{\mu} = 
	z_{\mu}(t)^{-1} \cdot z_{\nu}(t)^{-1} 
\end{gather}
whenever $\nu \cup \mu = \lambda$. Using \eqref{e: zee} and the formulas for product and coproduct in the power sum basis, we deduce that
\[
     p^*_\mu \cdot p^*_\nu = \binom{\mu\cup\nu}{\mu} p^*_{\mu\cup\nu}\,,
\qand
    \Delta(p^*_\lambda) = \sum_{\substack{\mu\leq \lambda \\ \mu\cup\nu = \lambda}} p^*_\mu \otimes p^*_\nu \,.
\]

\smallskip\noindent\emph{Checking the desired identities.} Using the preceding formulas, we get
\[
	\innp(\Delta(p_\lambda) , p^*_\mu \otimes p^*_\nu) \ = \ \binom{\lambda}{\mu} \cdot \delta_{\lambda,\mu\cup\nu}
	\ = \ \innp(p_\lambda , p^*_\mu \cdot p^*_\nu).
\]
and
\[
	\innp(p_\mu \cdot p_\nu,p^*_\lambda) \ = \  \delta_{\lambda,\mu\cup\nu}
	\ = \ \innp(p_\mu\otimes p_\nu,\Delta(p^*_\lambda{)}).
\]
This completes the proof of the lemma.
\end{proof}

After \eqref{e: skew-via-duality}, \eqref{e: harpoon} and Lemma \ref{l: self-dual}, we see that $P_{\lambda/\mu} = Q_\mu\harpoon P_\lambda$ and $Q_{\lambda/\mu} = P_\mu \harpoon Q_\lambda$.

\section{Proofs of the main theorems}\label{s: skew proofs}

We specialize \eqref{e: Hopf skew rule} to Hall--Littlewood polynomials, putting $a\harpoon g = P_{\lambda/\mu}$.

\begin{proof}[Proof of Theorem \ref{thm: e-Pieri}]
Taking $h = e_r$ in \eqref{e: Hopf skew rule}, we get
\begin{align}
\label{e: eskew1}
	P_{\lambda/\mu} \cdot e_r &= \left(Q_{\mu} \harpoon P_\lambda \right) \cdot e_r = \sum_{(e_r)} \left( S({e_r}'')\harpoon Q_\mu\right) \harpoon \bigl(P_\lambda \cdot {e_r}'\bigr) \\[0ex]
\label{e: eskew2}
	&= \sum_{k=0}^r \left( S(e_k)\harpoon Q_\mu\right) \harpoon \bigl(P_\lambda \cdot e_{r-k}\bigr) \\[.25ex]
\label{e: eskew3}
	&= \sum_{k=0}^r (-1)^k \left( s_k \harpoon Q_\mu\right) \harpoon \bigl(P_\lambda \cdot e_{r-k}\bigr) \\[.25ex]
\label{e: eskew4}
	&= \sum_{k=0}^r (-1)^k \biggl( \sum_\tau t^{n(\tau)} Q_{\mu/\tau} \biggr)\harpoon \bigl(P_\lambda \cdot e_{r-k}\bigr) \\[.25ex]
\label{e: eskew5}
	&= \sum_{k=0}^r (-1)^k \Biggl(\sum_{|\mu/\mu^-| = k} \biggl(\sum_\tau t^{n(\tau)} f_{\mu^-,\tau}^{\,\mu}(t)\biggr) Q_{\mu^-}\Biggr) \harpoon \Biggl(\sum_{|\lambda^+/\lambda| = r-k} \vs_{\lambda^+/\lambda}(t) {P_{\lambda^+}} \Biggr)\\[.25ex]
\label{e: eskew6}
	&= \sum_{\lambda^+,\mu^-} (-1)^{|\mu/\mu^-|} \sk_{\mu/\mu^-}(t) \vs_{\lambda^+/\lambda}(t) P_{\lambda^+/\mu^-} \,.
\end{align}

For \eqref{e: eskew2} and \eqref{e: eskew3}, we used Proposition \ref{p: esq}. For \eqref{e: eskew4}, we expanded $s_k$ in the $P$ basis (cf. the proof of Corollary \ref{cor: P.s}) and used the Hopf characterization of skew elements. Explicitly, 
$$
   s_k \harpoon Q_\mu = \biggl(\sum_{\tau \vdash k} t^{n(\tau)} P_\tau\biggr) \harpoon Q_\mu = \sum_{\tau \vdash k} t^{n(\tau)} Q_{\mu/\tau}\,.
$$  
We use \eqref{e: vs} and \eqref{e: skew-via-duality} to pass from \eqref{e: eskew4} to \eqref{e: eskew5}:  the coefficient of $Q_{\mu^-}$ in the expansion of $Q_{\mu/\tau}$ is equal to the coefficient of $P_\mu$ in $P_{\mu^-}P_\tau$. Finally, \eqref{e: eskew6} follows from Corollary \ref{cor: P.s}.
\end{proof}

\begin{proof}[Proof of Theorem \ref{thm: s-Pieri}]

Taking $h = s_r$ in \eqref{e: Hopf skew rule}, we get
\begin{align}
\label{e: sskew1}
	P_{\lambda/\mu} \cdot s_r &= \left(Q_{\mu} \harpoon P_\lambda \right) \cdot s_r = \sum_{(s_r)} \left( S({s_r}'')\harpoon Q_\mu\right) \harpoon \bigl(P_\lambda \cdot {s_r}'\bigr) \\[0ex]
\label{e: sskew2}
	&= \sum_{k=0}^r \left( S(s_k)\harpoon Q_\mu\right) \harpoon \bigl(P_\lambda \cdot s_{r-k}\bigr) \\[.25ex]
\label{e: sskew3}
	&= \sum_{k=0}^r (-1)^k \left( e_k \harpoon Q_\mu\right) \harpoon \bigl(P_\lambda \cdot s_{r-k}\bigr) \\[.25ex]
\label{e: sskew4}
	&= \sum_{k=0}^r (-1)^k Q_{\mu/1^k} \harpoon \bigl(P_\lambda \cdot s_{r-k}\bigr) \\[.25ex]
\label{e: sskew5}
	&= \sum_{k=0}^r (-1)^k \Biggl(\sum_{|\mu/\mu^-| = k} \vs_{\mu/\mu^-}(t) Q_{\mu^-}\Biggr) \harpoon \Biggl(\sum_{|\lambda^+/\lambda| = r-k} \sk_{\lambda^+/\lambda}(t) {P_{\lambda^+}} \Biggr) \\[.25ex]
\label{e: sskew6}
	&= \sum_{\lambda^+,\mu^-} (-1)^{|\mu/\mu^-|} \vs_{\mu/\mu^-}(t) \sk_{\lambda^+/\lambda}(t) P_{\lambda^+/\mu^-} \,.
\end{align}

For \eqref{e: sskew2} and \eqref{e: sskew3}, the proof is the same as above. For \eqref{e: sskew4}, we used $e_k = P_{1^k}$, while for \eqref{e: sskew5}, we used \eqref{e: vs} and \eqref{e: P.s}. Equation \eqref{e: sskew6} is obvious.
\end{proof}

\begin{proof}[Proof of Theorem \ref{thm: q-Pieri}]

We present two proofs. The first is along the lines of the preceding proofs of Theorems \ref{thm: e-Pieri} and \ref{thm: s-Pieri}. Taking $h=s_r$ in \eqref{e: Hopf skew rule}, we get
\begin{align}
\label{e: qskew1}
	P_{\lambda/\mu} \cdot q_r &= \left(Q_{\mu} \harpoon P_\lambda \right) \cdot q_r = \sum_{(q_r)} \left( S({q_r}'')\harpoon Q_\mu\right) \harpoon \bigl(P_\lambda \cdot {q_r}'\bigr) \\[0ex]
\label{e: qskew2}
	&= \sum_{k=0}^r \left( S(q_k)\harpoon Q_\mu\right) \harpoon \bigl(P_\lambda \cdot q_{r-k}\bigr) \\[.25ex]
\label{e: qskew3}
	&= \sum_{k=0}^r \left( \sum_{\tau \vdash k} c_\tau(t) P_\tau \harpoon Q_\mu\right) \harpoon \bigl(P_\lambda \cdot q_{r-k}\bigr) \\[.25ex]
\label{e: qskew4}
	&= \sum_{k=0}^r \left( \sum_{\tau \vdash k} c_\tau(t) Q_{\mu/\tau} \right)\harpoon \bigl(P_\lambda \cdot q_{r-k}\bigr) \\[.25ex]
\label{e: qskew5}
	&= \sum_{k=0}^r \left(\sum_{|\mu/\mu^-| = k} \left(\sum_\tau c_\tau(t) f_{\mu^-,\tau}^{\,\mu}(t)\right) Q_{\mu^-}\right) \harpoon \left(\sum_{|\lambda^+/\lambda| = r-k} \hs_{\lambda^+/\lambda}(t) {P_{\lambda^+}} \right)\\[.5ex]
\label{e: qskew6}
	&= \sum_{\lambda^+,\mu^-} (-1)^{|\mu/\mu^-|} (-t)^{|\tau/\mu^-|} \vs_{\mu/\tau}(t) \sk_{\tau/\mu^-} \hs_{\lambda^+/\lambda}(t) P_{\lambda^+/\mu^-} \,.
\end{align}
The only line that needs a comment is \eqref{e: qskew6}. 

Substitute $y = -1/t$, $\lambda = \mu$, $\mu = \mu^-$ and $\nu = \tau$ into Theorem \ref{thm: y}. We get
$$
	\sum_\tau \vs_{\mu/\tau}(t) \sk_{\tau/\mu^-}(t) (-1/t)^{|\mu/\tau|} = 
	\sum_{\sigma} t^{n(\sigma)-\binom{\ell(\sigma)}2} f_{\tau,\mu^-}^\mu(t) \prod_{j=1}^{\ell(\sigma)} (-1/t+t^{j-1}),
$$
and, after multiplying by $t^{|\mu/\mu^-|}$,
$$
	\sum_\tau (-1)^{|\mu/\tau|} t^{|\tau/\mu^-|} \vs_{\mu/\tau}(t) \sk_{\tau/\mu^-}(t) = 
	\sum_\sigma t^{n(\sigma) - \binom{\ell(\sigma)}2 + |\mu/\mu^-| - \ell(\sigma)} f_{\tau,\mu^-}^\mu(t) \prod_{j=1}^{\ell(\sigma)} (-1+t^{j}).
$$
Now $|\mu/\mu^-| = |\sigma|$ and $n(\sigma) - \binom{\ell(\sigma)}2 + |\sigma| - \ell(\sigma) = \sum_i (\binom{\sigma'_i}2 + \sigma'_i) - \binom{\sigma_1'+1}2 = \sum_{i = 2}^{\sigma_1} \binom{\sigma'_i+1}2$, which shows that
$$
	\sum_\sigma c_\sigma f_{\sigma,\mu^-}^\mu(t) = 
	\sum_\tau (-1)^{|\mu/\tau|} t^{|\tau/\mu^-|} \vs_{\mu/\tau}(t) \sk_{\tau/\mu^-}(t),
$$
with the sum over all $\tau$ satisfying $\mu^- \subseteq \tau \subseteq \mu$. This completes the first proof.

The second proof uses Theorems \ref{thm: P.s}, \ref{thm: e-Pieri} and \ref{thm: s-Pieri}. Recall from \eqref{qr} that $q_r = \sum_{k=0}^r (-t)^k s_{r-k} e_k$. We have
\begin{flalign*}
\phantom{==}P_{\lambda/\mu} \cdot q_r  &= P_{\lambda/\mu} \cdot \left(\sum_{k=0}^r (-t)^k s_{r-k} e_k\right) = \sum_{k=0}^r (-t)^k ( P_{\lambda/\mu} s_{r-k})e_k &\\
	 & = \sum_{k=0}^r (-t)^k  \sum_{\sigma,\tau} (-1)^{|\mu/\tau|}\vs_{\mu/\tau}(t) \sk_{\sigma/\lambda}(t) P_{\sigma/\tau} e_k 
\end{flalign*}
\begin{flalign*}
	 & = \!\!\!\sum_{\sigma,\tau,\mu^-,\lambda^+} \!\!\!\!(-t)^{|\tau/\mu^-|+|\lambda^+/\sigma|} (-1)^{|\mu/\tau| + |\tau/\mu^-|}\vs_{\mu/\tau}(t) \sk_{\sigma/\lambda}(t) \sk_{\tau/\mu^-}(t) \vs_{\lambda^+/\sigma}(t) P_{\lambda^+/\mu^-} \\
	 & = \!\!\!\sum_{\tau,\mu^-,\lambda^+} \!\!\!(-1)^{|\mu/\mu^-|} (-t)^{|\tau/\mu^-|}\vs_{\mu/\tau}(t)\sk_{\tau/\mu^-}(t) \!\left( \sum_\sigma (-t)^{|\lambda^+/\sigma|} \vs_{\lambda^+/\sigma}(t) \sk_{\sigma/\lambda}(t) \!\right) \!P_{\lambda^+/\mu^-} \\
	 & = \!\!\!\sum_{\tau,\mu^-,\lambda^+} \!\!\!(-1)^{|\mu/\mu^-|}(-t)^{|\tau/\mu^-|}\vs_{\mu/\tau}(t)\sk_{\tau/\mu^-}(t) \hs_{\lambda^+/\lambda}(t) P_{\lambda^+/\mu^-},
\end{flalign*}
where we used Lemma \ref{l: hs} in the final step.
\end{proof}

Our final result is on the uniqueness of the expansions.

\begin{theorem} \label{thm: unique}
 Let $a_{\lambda/\mu}(t)$ and $b_{\lambda/\mu}(t)$ be polynomials defined for $\lambda \supseteq \mu$, with $b_{\emptyset/\emptyset}(t) = 1$. For fixed $\lambda\supseteq\mu$ and $r\geq0$, consider the expression
$$
   \mathcal E_{\lambda,\mu,r} \ = \!\sum_{\substack{\lambda^+\supseteq\lambda,\,\mu^-\subseteq\mu\\ |\lambda^+/\lambda| + |\mu/\mu^-| = r}} (-1)^{|\mu/\mu^-|} a_{\lambda^+/\lambda}(t) b_{\mu/\mu^-}(t) P_{\lambda^+/\mu^-}.
$$
1)\ 
If $\mathcal E_{\lambda,\mu,r} = P_{\lambda/\mu} \, s_{1^r}$ $\forall\lambda,\mu,r$ then $a_{\lambda^+/\lambda} = \vs_{\lambda^+/\lambda}$ and $b_{\mu/\mu^-} = \sk_{\mu/\mu^-}$.  \\
2)\ 
If $\mathcal E_{\lambda,\mu,r} = P_{\lambda/\mu} \, s_{r}$ $\forall\lambda,\mu,r$ then $a_{\lambda^+/\lambda} = \sk_{\lambda^+/\lambda}$ and $b_{\mu/\mu^-} = \vs_{\mu/\mu^-}$. \\
3)\ 
If $\mathcal E_{\lambda,\mu,r} = P_{\lambda/\mu} \, q_r$ $\forall\lambda,\mu,r$ then $a_{\lambda^+/\lambda} = \hs_{\lambda^+/\lambda}$ and $b_{\mu/\mu^-} = \sum_{\nu} (-t)^{|\nu/\mu^-|} \vs_{\mu/\nu}\, \sk_{\nu/\mu^-}$.
\end{theorem}

\begin{proof}
We prove only the first statement, the others being similar. Suppose that we have
$$
   P_{\lambda/\mu} \, s_{1^r} = \sum_{\lambda^+,\mu^-} 
   (-1)^{|\mu/\mu^-|} a_{\lambda^+/\lambda}(t) b_{\mu/\mu^-}(t) P_{\lambda^+/\mu^-}\,.
$$
If we set $\mu = \emptyset$, we get the expansion of $P_{\lambda} s_{1^r}$ over (non-skew) Hall-Littlewood polynomials, which is, of course, unique. Therefore $a_{\lambda/\mu}(t) \, b_{\emptyset/\emptyset}(t) =a_{\lambda/\mu}(t) = \vs_{\lambda/\mu}(t)$ for all $\lambda\supseteq\mu$. We will prove by induction on $|\lambda/\mu|$ that $b_{\lambda/\mu}(t) = \sk_{\lambda/\mu}(t)$. For $\lambda = \mu$ and $r = 0$, we get $P_{\lambda/\lambda} = b_{\lambda/\lambda}(t) P_{\lambda/\lambda}$, so $b_{\lambda/\lambda}(t) = 1 = \sk_{\lambda/\lambda}(t)$. Suppose that $b_{\lambda/\mu}(t) = \sk_{\lambda/\mu}(t)$ for $|\lambda/\mu| < r$ and that $|\lambda/\mu| = r$. Take
 \begin{align*}
  \sigma &= (\underbrace{\lambda_1 + \mu_1, \ldots, \lambda_1 + \mu_1}_{\ell(\lambda)},\lambda_1 + \mu_1,\lambda_1+\mu_2,\ldots,\lambda_1+\mu_{\ell(\mu)}) \\
  \tau &= (\underbrace{\lambda_1 + \mu_1, \ldots, \lambda_1 + \mu_1}_{\ell(\lambda)},\underbrace{\lambda_1,\ldots,\lambda_1}_{\ell(\mu)}).
 \end{align*}

 Note that $\lambda \subseteq \sigma$. Also, the diagram of $\sigma/\tau$ is a translation of the diagram of $\mu$. That means there is only one LR-sequence $S$ (see \cite[p.\ 185]{Mac}) of shape $\sigma/\tau$, and it has type $\mu$. This implies that $f^\sigma_{\tau,\mu} = f_S(t)$, $f^\sigma_{\tau,\mu'} = 0$ for $\mu \neq \mu'$ (see \cite[pp.\ 194 and 218]{Mac}). Therefore $P_{\sigma/\tau}$ is a non-zero polynomial multiple of $P_\mu$.
 Now 
 \begin{align*}
 P_{\sigma/\lambda} \, s_{1^r} &= \sum_{\sigma^+,\lambda^-} (-1)^{|\lambda/\lambda^-|} a_{\sigma^+/\sigma}(t) b_{\lambda/\lambda^-}(t) P_{\sigma^+/\lambda^-} \\
 &= \sum_{\sigma^+,\lambda^-} (-1)^{|\lambda/\lambda^-|} \vs_{\sigma^+/\sigma}(t) b_{\lambda/\lambda^-}(t) P_{\sigma^+/\lambda^-} \\
 &= \sum_{\sigma^+,\lambda^-} (-1)^{|\lambda/\lambda^-|} \vs_{\sigma^+/\sigma}(t) \sk_{\lambda/\lambda^-}(t) P_{\sigma^+/\lambda^-},
 \end{align*}
 where we used Theorem \ref{thm: e-Pieri}. By the induction hypothesis, $b_{\lambda/\lambda^-}(t) = \sk_{\lambda/\lambda^-}(t)$ if $|\lambda/\lambda^-| < r$. After cancellations, we get
 $$\sum_{\lambda^-} (-1)^{|\lambda/\lambda^-|} (b_{\lambda/\lambda^-}(t) - \sk_{\lambda/\lambda^-}(t)) P_{\sigma/\lambda^-} = 0,$$
 where the sum on the left is over all $\lambda^- \subseteq \lambda$ such that $|\lambda/\lambda^-| = r$. Now take scalar product with $Q_\tau$. Since $\langle P_{\sigma/\lambda^-},Q_\tau \rangle = \langle P_\sigma, Q_{\lambda^-}Q_\tau \rangle = \langle P_{\sigma/\tau},Q_{\lambda^-} \rangle $ is the coefficient of $P_{\lambda^-}$ in $P_{\sigma/\tau}$, we see that $(-1)^{|\lambda/\mu|} (b_{\lambda/\mu}(t) - \sk_{\lambda/\mu}(t)) = 0$. That is, $b_{\lambda/\mu}(t) = \sk_{\lambda/\mu}(t)$. 
\end{proof}

\begin{remark}
Similar proofs show that the expansions of $s_{\lambda/\mu} s_{1^r}$, $s_{\lambda/\mu} s_r$ and $s_{\lambda/\mu} P_r$ in terms of skew Schur functions are also unique in the sense of Theorem \ref{thm: unique}, a fact that was not noted in either \cite{AM} or \cite{K}.
\end{remark}

\begin{remark}
It would be preferable to have a simpler expression for the polynomial
 \begin{equation} \label{e: d1}
  b_{\lambda/\mu}(t)  = \sum_{\nu} (-t)^{|\nu/\mu|}\vs_{\lambda/\nu}(t)\sk_{\nu/\mu}(t) 
 \end{equation}
from Theorems \ref{thm: q-Pieri} and \ref{thm: unique}(3), i.e., one involving only the boxes of $\lambda/\mu$ in the spirit of $\hs_{\lambda/\mu}(t)$, so that we could write
$$
 	P_{\lambda/\mu} \cdot q_r =
	\sum_{\lambda^+,\mu^-} {(-1)^{|\mu/\mu^-|}} \hs_{\lambda^+/\lambda}(t) b_{\mu/\mu^-}(t)\, P_{\lambda^+/\mu^-}\,,
$$ 
where the sum is over all $\lambda^+ \supseteq \lambda$, $\mu^-\subseteq \mu$ such that $|\lambda^+/\lambda| + |\mu/\mu^-| = r$. 

Toward this goal, we point out a hidden symmetry in the polynomials $b_{\lambda/\mu}$(t). Writing $q_r$ as $\sum_{k=0}^r (-t)^ke_k s_{r-k}$ before running through the second proof of Theorem \ref{thm: q-Pieri} (i.e., applying Theorems \ref{thm: e-Pieri} and \ref{thm: s-Pieri} in the reverse order) reveals
\begin{equation}\label{e: d2}
	b_{\lambda/\mu}(t) = \sum_{\nu} (-t)^{|\lambda/\nu|} \sk_{\lambda/\nu}(t)\, \vs_{\nu/\mu}(t)\,.
\end{equation}
%
Further toward this goal, 
note how similar \eqref{e: d1} is to the sum in Lemma \ref{l: hs}, which reduces to the tidy product of polynomials $\hs_{\lambda/\mu}(t)$. 
 
Basic computations suggest some hint of a polynomial-product description for $b_{\lambda/\mu}(t)$,
\begin{align*}
  \skewshape[6]{4,3,3,1}{3,2,2,1} \ \, 
  	:& \ \  -(t-1)^2 (t+1) \left(t^3+t^2+t-1\right) \\[.25ex]
  \skewshape[6]{4,3,3,2}{3,2,2,1} \ \, 
  	:& \ \ (t-1)^2 (t+1) \left(t^3+t^2+t-1\right)^2 \\[.25ex]
  \skewshape[6]{5,3,3,2}{3,2,2,1} \ \, 
  	:& \ \ t (t-1)^2 (t+1) \left(t^3+t^2+t-1\right)^2 \\[.25ex]
  \skewshape[6]{5,3,3,2,1}{3,2,2,1} \ \, 
  	:& \ \ t (t-1)^2 (t+1) \left(t^2+t-1\right) \left(t^3+t^2+t-1\right)^2 , 
\intertext{but others suggest that such a description will not be tidy,}
  \skewshape[6]{5,3,3,2,2}{3,2,2,1} \ \, 
  	:& \ \ -t^2 (t-1)^2 (t+1)^2 \left(t^3+t^2+t-1\right) \left(t^7+t^6+2 t^5-t^3-2   t^2-t+1\right) .
\end{align*}
We leave a concise description of the $b_{\lambda/\mu}(t)$ as an open problem. 
\end{remark}

\nocite{Mont,St}
\bibliographystyle{plain}  

\begin{thebibliography}{10}

\bibitem{Akin}
Kaan Akin.
\newblock On complexes relating the {J}acobi-{T}rudi identity with the
  {B}ernstein-{G}el\cprime fand-{G}el\cprime fand resolution.
\newblock {\em J. Algebra}, 117(2):494--503, 1988.

\bibitem{AM2}
Sami~H. Assaf and Peter R.~W. McNamara.
\newblock A {P}ieri rule for skew shapes.
\newblock Slides for a talk at FPSAC 2010, available at
  \texttt{http://linux.bucknell.edu/{\textasciitilde}pm040/Slides/McNamara.pdf%
}.

\bibitem{AM}
Sami~H. Assaf and Peter R.~W. McNamara.
\newblock A {P}ieri rule for skew shapes.
\newblock {\em J. Combin. Theory Ser. A}, 118(1):277--290, 2011.
\newblock {W}ith an appendix by {T}homas {L}am.

\bibitem{Doty}
S.~R. Doty.
\newblock Resolutions of {$B$} modules.
\newblock {\em Indag. Math. (N.S.)}, 5(3):267--283, 1994.

\bibitem{Fu}
William Fulton.
\newblock {\em Young tableaux}, volume~35 of {\em London Mathematical Society
  Student Texts}.
\newblock Cambridge University Press, Cambridge, 1997.

\bibitem{K}
Matja{\v{z}} Konvalinka.
\newblock Skew quantum {M}urnaghan--{N}akayama rule.
\newblock To appear in \emph{J.\ Algebraic Combin.}, arXiv:1101.5250.

\bibitem{LLS}
Thomas Lam, Aaron Lauve, and Frank Sottile.
\newblock Skew {L}ittlewood--{R}ichardson rules from {H}opf algebras.
\newblock {\em Int. Math. Res. Notices}, 2011:1205--1219, 2011.

\bibitem{LLT}
Alain Lascoux, Bernard Leclerc, and Jean-Yves Thibon.
\newblock Ribbon tableaux, {H}all-{L}ittlewood functions, quantum affine
  algebras, and unipotent varieties.
\newblock {\em J. Math. Phys.}, 38(2):1041--1068, 1997.

\bibitem{Mac}
I.~G. Macdonald.
\newblock {\em Symmetric functions and {H}all polynomials}.
\newblock Oxford Mathematical Monographs. The Clarendon Press Oxford University
  Press, New York, second edition, 1995.

\bibitem{Mont}
Susan Montgomery.
\newblock {\em Hopf algebras and their actions on rings}, volume~82 of {\em
  CBMS Regional Conference Series in Mathematics}.
\newblock Published for the Conference Board of the Mathematical Sciences,
  Washington, DC, 1993.

\bibitem{St}
Richard~P. Stanley.
\newblock {\em Enumerative combinatorics. {V}ol. 2}, volume~62 of {\em
  Cambridge Studies in Advanced Mathematics}.
\newblock Cambridge University Press, Cambridge, 1999.

\bibitem{Zel:1981}
A.~V. Zelevinsky.
\newblock A generalization of the {L}ittlewood-{R}ichardson rule and the
  {R}obinson-{S}chensted-{K}nuth correspondence.
\newblock {\em J. Algebra}, 69(1):82--94, 1981.

\bibitem{Zel:1987}
A.~V. Zelevinsky.
\newblock Resolutions, dual pairs and character formulas.
\newblock {\em Functional Anal. Appl.}, 21(2):152--154, 1987.

\end{thebibliography}
\def\cprime{$'$}

\end{document}